\documentclass[11pt]{amsart}

\usepackage{amssymb,amsmath,amscd,amsrefs,color,url}
\usepackage{hyperref,MnSymbol,enumitem}

\newcommand{\A}{\mathbb{A}}
\newcommand{\Q}{\mathbb{Q}}

\newcommand{\Z}{\mathbb{Z}}
\newcommand{\F}{\mathbb{F}}
\renewcommand{\P}{\mathfrak{P}}
\newcommand{\p}{\mathfrak{p}}
\renewcommand{\q}{\mathfrak{q}}

\newcommand{\calC}{\mathcal C}
\newcommand{\calD}{\mathcal D}
\newcommand{\calE}{\mathcal E}
\newcommand{\calF}{\mathcal F}
\newcommand{\calO}{\mathcal O}
\newcommand{\calP}{\mathcal P}
\newcommand{\calR}{\mathcal R}

\renewcommand{\to}{\rightarrow}
\newcommand{\disc}{\operatorname{disc}}

\newcommand{\Gal}{\operatorname{Gal}}
\newcommand{\Spec}{\operatorname{Spec}}

\newtheorem*{thm*}{Theorem}
\newtheorem{thm}{Theorem}[section]
\newtheorem{lem}[thm]{Lemma}
\newtheorem{prop}[thm]{Proposition}

\theoremstyle{definition}
\newtheorem{rem}[thm]{Remark}
\newtheorem{alg}[thm]{Algorithm}

\numberwithin{equation}{section}

\begin{document}

\title[Explicit Hilbert Irreducibility]{Galois groups over rational function fields and Explicit Hilbert Irreducibility}

\author[D. Krumm]{David Krumm}
\address{Mathematics Department\\
Reed College}
\email{dkrumm@reed.edu}
\urladdr{http://maths.dk}

\author[N. Sutherland]{Nicole Sutherland}
\address{Computational Algebra Group\\
School of Mathematics and Statistics\\
The University of Sydney}
\email{nicole.sutherland@sydney.edu.au}


\begin{abstract}
Let $P\in\mathbb Q[t,x]$ be a polynomial in two variables with rational coefficients, and let $G$ be the Galois group of $P$ over the field $\mathbb Q(t)$. It follows from Hilbert's Irreducibility Theorem that for most rational numbers $c$ the specialized polynomial $P(c,x)$ has Galois group isomorphic to $G$ and factors in the same way as $P$. In this paper we discuss methods for computing the group $G$ and obtaining an explicit description of the exceptional numbers $c$, i.e.,  those for which $P(c,x)$ has Galois group different from $G$ or factors differently from $P$. To illustrate the methods we determine the exceptional specializations of three sample polynomials. In addition, we apply our techniques to prove a new result in arithmetic dynamics.
\end{abstract}

\maketitle

\section{Introduction}\label{intro_section} 

Let $P\in \Q[t,x]$ be nonconstant in the variable $x$, and let $G$ be the Galois group of $P$ over the function field $\Q(t)$. For any rational number $c$ we may consider the polynomial $P_c=P(c,x)\in\Q[x]$ and its Galois group, which we denote by $G_c$. Hilbert's Irreducibility Theorem (henceforth abbreviated HIT) implies that as $c$ varies in $\Q$, most specializations $P_c$ have Galois group isomorphic to $G$ and factor in the same way as $P$. However, there may exist rational numbers $c$ such that $G_c$ is not isomorphic to $G$ or $P_c$ factors differently from $P$; we will call the set of all such numbers the \textit{exceptional set} of $P$, denoted $\calE(P)$. The main purpose of this article is to develop a method for obtaining an explicit description of the set $\calE(P)$.

A standard step in the proof of HIT is to show that there exist a finite set $D\subset\Q$ and algebraic curves $C_1,\ldots, C_r$ having the following property: for $c\in\Q\setminus D$, $c$ belongs to the set $\calE(P)$ if and only if $c$ is a coordinate of a rational point on one of the curves $C_i$. Our method for explicitly describing the exceptional set of $P$ is based on a constructive proof of this result, which we summarize in the following theorem.

\begin{thm*}[see \ref{main_hit_thm}]\label{main_hit_thm_intro}
Let $\Delta(t)$ and $\ell(t)$ be the discriminant and leading coefficient of $P$, respectively. Let $M_1,\ldots, M_r$ be representatives of all the conjugacy classes of maximal subgroups of $G$. For $i=1,\ldots, r$, let $F_i$ be the fixed field of $M_i$ and let $f_i(t,x)$ be a monic irreducible polynomial in $\Q[t][x]$ such that $F_i/\Q(t)$ is generated by a root of $f_i(t,x)$. Suppose that $c\in \Q$ satisfies
\begin{equation}\label{intro_exceptional_set_eq}
\Delta(c)\cdot\ell(c)\cdot\prod_{i=1}^r\disc f_i(c,x)\ne 0.
\end{equation}
Then $c\in\calE(P)$ $\iff$ there is an index $i$ such that $f_i(c,x)$ has a root in $\Q$.
\end{thm*}

It should be noted that versions of this theorem can be found in the literature; see, for instance, \cite[\S3.1]{debes-walkowiak}. However, we are not aware of any reference proving the result for reducible polynomials $P$, or explicitly identifying a finite set that needs to be excluded, as in \eqref{intro_exceptional_set_eq}.

It follows from the theorem that we may take $D$ to be the set of all $c\in\Q$ for which \eqref{intro_exceptional_set_eq} does not hold, and we may take $C_i$ to be the plane curve defined by the equation $f_i(t,x)=0$. The problem of explicitly describing the set $\calE(P)$ can therefore be reduced to the following: 
\begin{enumerate}
\item Compute the polynomials $f_i$, and
\item Determine all the rational points on the curves $C_i$.
\end{enumerate}

At present there is no hope of an algorithmic solution to the problem of determining the set of rational points on an algebraic curve, though there are several techniques available for approaching the problem on an \emph{ad hoc} basis. A survey of current methods may be found in \cite{stoll_survey}. We give in \S\ref{examples_section} a few relatively simple examples using some of these methods; more sophisticated examples can be found in the articles \cite{poonen} and \cite{flynn-poonen-schaefer}.

In contrast to step (2) above, all of the computational tools needed for the first step are currently available. To achieve step (1), one must begin by computing a permutation representation of the Galois group $G$; this can be done using methods of Fieker and Kl\"uners \cite{fieker-kluners}. Though these authors mainly discuss the case of irreducible polynomials over $\Q$, their methods can be extended to work more generally. For instance, Fieker \cite{fieker_implementation} adapted the algorithm to compute Galois groups of irreducible polynomials over $\Q(t)$. In the present paper we discuss the modifications needed for Fieker's implementation and we further extend the method so that it applies to reducible polynomials over $\Q(t)$. This generalized algorithm for computing Galois groups over $\Q(t)$ has been implemented by the second author and is included in \textsc{Magma} V2.24 \cite{bosma-cannon-playoust}.

Once the group $G$ has been computed, its maximal subgroups can be obtained using an algorithm of Cannon and Holt \cite{cannon-holt}. Finally, the fixed field of any subgroup of $G$ can be computed using known methods; see \cite[\S 3.3]{kluners-malle} and our discussion in \S\ref{fixed_field_section}. Hence, given the polynomial $P$ it is possible to compute defining equations for the curves $C_i$. Functionality for this computation is available in \textsc{Magma} via the intrinsic \texttt{HilbertIrreducibilityCurves}.

In summary, by using currently available methods in computational group theory and Galois theory, and by applying techniques for determining rational points on curves, it is possible in many cases to obtain a complete characterization of the exceptional set of a given polynomial $P\in\Q[t,x]$.

This article is organized as follows. We devote \S\ref{hit_section} to the proof of the above theorem, and \S\ref{algorithm_section} to a discussion of the algorithms for computing Galois groups and fixed fields over $\Q(t)$. In order to illustrate the process described earlier, we include three examples in \S\ref{examples_section}.

The first example concerns the polynomial $P(t,x)=x^6+t^6-1$, which has a finite exceptional set. The case $n=3$ of Fermat's Last Theorem implies that the only rational numbers $c$ for which $P_c$ has a rational root are 0 and $\pm 1$. We prove that in fact 0 and $\pm 1$ are the only rational numbers for which $P_c$ is reducible.

In the second example we consider the polynomial $P(t,x)=x^6 - 4x^2 - t^2$, which is irreducible and has Galois group isomorphic to the symmetric group $S_4$. Our analysis will show that, in addition to the obvious reducible specialization $P_0$, there is an infinite family of reducible specializations. More precisely, we prove that for $c\ne 0$,
\[P_c\;\text{is reducible if and only if}\;\; c=\frac{v^4+16}{8v}\;\;\text{for some}\;\;v\in\Q.\]

Moreover, when $c$ has the above form we show that $P_c$ factors as a product of two irreducible cubic polynomials.

The third example relates to the polynomial $P(t,x)=3x^4-4x^3+1+3t^2$, which is one polynomial in a family discussed by Serre \cite[\S4.5]{serre_topics}. The Galois group of $P$ is isomorphic to the alternating group $A_4$, so a typical specialization $P_c$ will have Galois group $G_c\cong A_4$. However, there are infinitely many exceptions to this: we prove that
\[G_c\not\cong A_4\iff c=\frac{v^3 - 9v}{9(1 - v^2)}\;\;\text{for some}\;v\in\Q.\]

Furthermore, for numbers $c$ of the above form, we determine precisely which groups $G_c$ arise as $v$ varies. We show in particular that the groups $(\Z/2\Z)\times(\Z/2\Z)$ and $\Z/2\Z$ arise for infinitely many such numbers $c$.

In addition to these detailed examples with polynomials $P$ of small degree, we briefly discuss two additional examples with polynomials of degrees 12 and 30.

In \S\ref{dynamics_section} we apply our methods to prove a new result in arithmetic dynamics. Let $\phi(z)\in\Q(z)$ be a rational function, and for $n\ge 1$ let $\phi^n$ denote the $n$-fold composition of $\phi$ with itself. We say that a rational number $x$ is \textit{periodic} for $\phi$ if there exists $n\ge 1$ such that $\phi^n(x)=x$; in that case, the least such $n$ is called the \textit{period} of $x$.  An important open problem in arithmetic dynamics is a uniform boundedness conjecture of Morton and Silverman \cite{morton-silverman} which in particular would imply the following: there exists a constant $M$ such that for every rational function $\phi(z)$ of degree 2 and every period $n>M$, $\phi$ has no rational point of period $n$. This conjecture has been refined in various special cases. For example, Poonen \cite{poonen} studied the family of maps of the form $\phi(z)=z^2+c$ and Manes \cite{manes} studied maps of the form $\phi(z)=kz+b/z$. Manes conjectures that no such map can have a rational point of period $n>4$, and shows that there exist at most finitely many such maps having a rational point of period 5. We prove the following stronger statement: for all but finitely many maps of the form $\phi(z)=kz+b/z$, there exists a positive proportion of prime numbers $p$ such that $\phi(z)$ does not have a point of period 5 in the $p$-adic field $\Q_p$.

\section{An explicit form of HIT}\label{hit_section}

Let $k$ be a field of characteristic 0 and let $P(t,x)\in k[t,x]$ be a polynomial of degree $n\ge 1$ in the variable $x$. We will henceforth regard $P$ as an element of the ring $k(t)[x]$ and assume that $P$ is separable. We define the \textit{factorization type} of $P$, denoted $\calF(P)$, to be the multiset consisting of the degrees of the irreducible factors of $P$. 

Let $N/k(t)$ be a splitting field of $P$ and let $G=\Gal(N/k(t))$ be the Galois group of $P$. We assume that $G$ is nontrivial. For every element $c\in k$, let $P_c$ denote the specialized polynomial $P(c,x)\in k[x]$. The Galois group and factorization type of $P_c$ will be denoted by $G_c$ and $\calF(P_c)$, respectively.

It follows from HIT that there is a thin\footnote{See \cite[\S 3.1]{serre_topics} for a definition of thin sets.} subset of $k$ outside of which we have $\calF(P_c)=\calF(P)$ and $G_c\cong G$. We define the \textit{exceptional set} of $P$, denoted $\calE(P)$, to be the set of all elements $c\in k$ for which either one of these conditions fails to hold:
\[\calE(P)=\{c\in k\;\vert\;\calF(P_c)\ne\calF(P)\;\text{or}\;G_c\not\cong G\}.\]
Our aim in this section is to prove a version of HIT from which one can deduce a method for explicitly describing the set $\calE(P)$; our main result in this direction is Theorem \ref{main_hit_thm} below.

It should be noted that the expert will be familiar with several of the results proved in this section. However, we have included complete proofs of most statements due to the lack of a reference treating this subject at the desired level of generality, in particular allowing the polynomial $P$ to be reducible.

Let $\Delta(t)$ and $\ell(t)$ be the discriminant and leading coefficient of $P$, respectively, and let $A\subset k(t)$ be the ring
\begin{equation}\label{A_defn}A=k[t]\left[\ell(t)^{-1}\right],\end{equation}
i.e., the localization of $k[t]$ by the multiplicative set generated by $\ell(t)$. 

For every intermediate field $F$ between $k(t)$ and $N$, let $\calO_F$ denote the integral closure of $A$ in $F$. Note that $\calO_F/A$ is an extension of Dedekind domains with $A$ being a PID. By a \textit{prime} of $F$ (or of $\calO_F)$ we mean a maximal ideal of $\calO_F$. If $\p$ is a prime of $A$ and $\q$ is a prime of $\calO_F$, we denote by $\kappa(\q)$ and $\kappa(\p)$ the residue fields of $\q$ and $\p$, respectively. Thus,
\[\kappa(\q)=\calO_F/\q,\;\;\kappa(\p)=A/\p.\]
If $\q$ divides $\p\calO_F$, we denote the ramification index and residual degree of $\q$ over $\p$ by $e(\q/\p)$ and $f(\q/\p)$, respectively.

For every prime $\P$ of $N$, let $G_{\P}$ be the decomposition group of $\P$ over $k(t)$ and let $Z_{\P}$ be the decomposition field of $\P$, i.e., the fixed field of $G_{\P}$. We refer the reader to \cite[Chap. I, \S\S 8-9]{neukirch} for the standard material on decomposition groups and ramification used in this section. 

If $c\in k$ is any element satisfying $\ell(c)\ne 0$, the evaluation homomorphism $k[t]\to k$ given by $a(t)\mapsto a(c)$ extends uniquely to a homomorphism $A\to k$. Let $\p_c$ be the kernel of this map. We will henceforth identify the residue field $\kappa(\p_c)$ with $k$ via the map $a(t)\bmod\p_c\mapsto a(c)$. Note that with this identification, if $f(t,x)\in A[x]$ is an arbitrary polynomial, then upon reducing the coefficients of $f$ modulo $\p_c$ we obtain the specialized polynomial $f(c,x)\in k[x]$.

It will be necessary for our purposes in this section to be able to determine how the prime $\p_c$ factors in any intermediate field $F$ between $k(t)$ and $N$. Recall that by a well-known theorem of Dedekind and Kummer, for all but finitely many primes $\p$ of $A$, the factorization of $\p$ in $F$ can be determined by choosing an integral primitive element $\theta$ of $F/k(t)$ and factoring its minimal polynomial modulo $\p$. The finite set of primes that need to be excluded are those that are not relatively prime to the conductor of the ring $A[\theta]$; see \cite[p. 47, Prop. 8.3]{neukirch} for details. The following lemma provides sufficient conditions on $c\in k$ so that $\p_c$ will be relatively prime to this conductor, and therefore the Dedekind--Kummer criterion can be applied to $\p_c$.

\begin{lem}\label{conductor_lem}
Let $F$ be an intermediate field between $k(t)$ and $N$ with primitive element $\theta\in\calO_F$ having minimal polynomial $f(t,x)\in A[x]$. Let
\[\mathfrak F=\{\alpha\in\calO_F\;\vert\;\alpha\cdot\calO_F\subseteq A[\theta]\}\]
be the conductor of the ring $A[\theta]$. Suppose that $c\in k$ satisfies
\[\ell(c)\cdot\disc f(c,x)\ne 0.\]
Then $\p_c\calO_F$ is relatively prime to $\mathfrak F$. Furthermore, $\p_c$ is unramified in $F$.
\end{lem}
\begin{proof}
Let $\delta\in A$ be the discriminant of $f$. By a linear algebra argument (see Lemma 2.9 in \cite[p. 12]{neukirch}) we have $\delta\cdot\calO_F\subseteq A[\theta]$ and therefore $\delta\in\mathfrak F$. Suppose that $\q$ is a prime of $F$ dividing both $\mathfrak F$ and $\p_c\calO_F$. Since $\mathfrak F\subseteq\q$ we have $\delta\in\q$, so $\delta\in\q\cap A=\p_c$. By definition of $\p_c$ this implies that $\disc f(c,x)=\delta(c)=0$, which is a contradiction. Therefore $\p_c$ must be relatively prime to $\mathfrak F$. 

The Dedekind--Kummer theorem now allows us to relate the factorization of $\p_c$ in $F$ to the factorization of $f(c,x)$ in $k[x]$. In particular, the theorem implies that if $\p_c$ is ramified in $F$, then $f(c,x)$ has a repeated irreducible factor, which contradicts our assumption that $\disc f(c,x)\ne 0$. Therefore $\p_c$ must be unramified in $F$.
\end{proof}

\begin{lem}\label{unramified_lem}
Suppose that  $c\in k$ satisfies $\Delta(c)\cdot\ell(c)\ne 0$. Then the prime $\p_c$ is unramified in $N$.
\end{lem}
\begin{proof}
Since $N$ is the compositum of the fields $k(t)(\theta)$ as $\theta$ ranges over all the roots of $P$ in $N$, it suffices to show that $\p_c$ is unramified in every such field. (See \cite[p. 119, Cor. 8.7]{lorenzini}.) Thus, let $\theta\in \calO_N$ be any root of $P$ and let $F=k(t)(\theta)$. Let $Q\in k[t][x]$ be an irreducible factor of $P$ having $\theta$ as a root. Dividing $Q$ by its leading coefficient we obtain a monic irreducible polynomial $f\in A[x]$ having $\theta$ as a root; it follows that $f$ is the minimal polynomial of $\theta$ over $k(t)$. Let $\delta\in A$ be the discriminant of $f$. Since $f$ divides $P$ in $A[x]$, $\delta$ divides $\Delta$ in $A$. Hence, the hypothesis that $\Delta(c)\ne 0$ implies that $\delta(c)\ne 0$. By Lemma \ref{conductor_lem}, $\p_c$ is unramified in $F$.
\end{proof}

We recall the notion of an isomorphism of group actions. If $\mathcal G$ and $\mathcal H$ are groups acting on sets $X$ and $Y$, respectively, then we say that there is an isomorphism of group actions between $\mathcal G$ and $\mathcal H$ if there exist an isomorphism $\phi:\mathcal G\to\mathcal H$ and a bijection $\sigma:X\to Y$ such that $\sigma(g\cdot x)=\phi(g)\cdot\sigma(x)$ for all $g\in\mathcal G$ and all $x\in X$.

\begin{prop}\label{decomposition_group_prop}
Suppose that $c\in k$ satisfies $\Delta(c)\cdot\ell(c)\ne 0$, and let $\P$ be a prime of $N$ dividing $\p_c$. Then there is an isomorphism of group actions $G_{\P}\cong G_c$, where $G_{\P}$ acts on the roots of $P$ and $G_c$ acts on the roots of $P_c$.
\end{prop}

\begin{proof}
For every element $a\in\calO_N$ let $\bar a$ denote the image of $a$ under the quotient map $\calO_N\to \kappa(\P)$. Recall that the extension $\kappa(\P)/k$ is Galois and that there is a surjective homomorphism $G_{\P}\to\Gal(\kappa(\P)/k)$ given by $\sigma\mapsto\bar\sigma$, where $\bar\sigma(\bar a)=\overline{\sigma(a)}$ for every $a\in\calO_N$. Furthermore, since $\p_c$ is unramified in $N$ by Lemma \ref{unramified_lem}, this map is an isomorphism. We claim that $\kappa(\P)$ is a splitting field for $P_c$.

Note that if $\alpha\in\calO_N$ is a root of $P$, then $\bar \alpha\in\kappa(\P)$ is a root of $P_c$. Moreover, if $\alpha$ and $\beta$ are distinct roots of $P$, then $\bar \alpha\ne\bar \beta$; indeed, this follows from the fact that $\bar\Delta=\Delta(c)\ne 0$. Thus, reduction modulo $\P$ is an injective map from the set of roots of $P$ to the set of roots of $P_c$.

Let $x_1,\ldots, x_n\in\calO_N$ be the roots of $P$ in $N$, and let $S=k(\bar x_1,\ldots, \bar x_n)$. Clearly $S$ is a splitting field for $P_c$, and $k\subseteq S\subseteq \kappa(\P)$. We will prove that $S=\kappa(\P)$ by showing that the group $\Gal(\kappa(\P)/S)$ is trivial. Let $\tau\in\Gal(\kappa(\P)/S)$ and let $\sigma\in G_{\P}$ be the element such that $\bar\sigma=\tau$. Since $\tau$ is the identity map on $S$, we have $\tau(\bar x_i)=\bar x_i$ for every index $i$, and hence $\overline{\sigma(x_i)}=\bar x_i$ for all $i$. Since $\sigma(x_i)$ and $x_i$ are roots of $P$, this implies that $\sigma(x_i)=x_i$. Thus, $\sigma$ fixes every root of $P$, so $\sigma$ is the identity element of $G_{\P}$. Hence $\tau=\bar\sigma$ is the identity element of $\Gal(\kappa(\P)/S)$. This proves that $\Gal(\kappa(\P)/S)$ is trivial and therefore $\kappa(\P)=S$ is a splitting field for $P_c$.

The map $\sigma\mapsto\bar\sigma$ is thus an isomorphism $G_{\P}\to\ G_c$. Moreover, the fact that $\sigma(x_i)=x_j\iff \bar\sigma(\bar x_i)=\bar x_j$ implies that the actions of $G_{\P}$ and $G_c$ are isomorphic.
\end{proof}

\begin{rem}
In the case where the polynomial $P$ is irreducible, Proposition \ref{decomposition_group_prop} follows from Theorem 2.9 in \cite[Chap. VII, \S2]{lang_algebra}.
\end{rem}

\begin{lem}\label{decomposition_field_lem}
Let $\p$ be a prime of $A$ and let $\P$ be a prime of $N$ dividing $\p$. Then the following hold:
\begin{enumerate}
\item Setting $\mathfrak Q=\P\cap Z_{\P}$, we have $e(\mathfrak Q/\p)=f(\mathfrak Q/\p)=1$.
\item Let $F$ be an intermediate field between $k(t)$ and $N$, and let $\q=\P\cap F$. If $e(\q/\p)=f(\q/\p)=1$, then $F\subseteq Z_{\P}$.
\end{enumerate}
\end{lem}
\begin{proof}
See \cite[p. 55, Prop. 9.3]{neukirch} and \cite[p. 118, Prop. 8.6]{lorenzini}.
\end{proof}

\begin{prop}\label{decomposition_field_prop}
Let $F$ be an intermediate field between $k(t)$ and $N$. Let $\theta\in\calO_F$ be a primitive element for $F/k(t)$ and let $f(t,x)\in A[x]$ be its minimal polynomial. Suppose that $c\in k$ satisfies 
\[\Delta(c)\cdot\ell(c)\cdot\disc f(c,x)\ne 0.\]
Then the following are equivalent:
\begin{enumerate}
\item The polynomial $f(c,x)$ has a root in $k$.
\item There exists a prime $\q$ of $F$ dividing $\p_c$ such that $f(\q/\p_c)=1$.
\item There exists a prime $\P$ of $N$ dividing $\p_c$ such that $F\subseteq Z_{\P}$.
\item There exists a prime $\P$ of $N$ dividing $\p_c$ such that $G_{\P}\subseteq \Gal(N/F)$.
\end{enumerate}
\end{prop}
\begin{proof}
By Lemma \ref{conductor_lem}, $\p_c$ is relatively prime to the conductor of $A[\theta]$. The Dedekind--Kummer theorem then implies that the degrees of the irreducible factors of $f(c,x)$ in $k[x]$ correspond to the residual degrees $f(\q/\p_c)$ for primes $\q$ of $F$ dividing $\p_c$. The equivalence of (1) and (2) follows immediately.

We now show that (2) and (3) are equivalent. Suppose that (2) holds, and let $\P$ be a prime of $N$ dividing $\q$. By Lemma \ref{unramified_lem}, $\p_c$ is unramified in $N$ and therefore unramified in $F$. Hence, $e(\q/\p_c)=1$. By Lemma \ref{decomposition_field_lem}, $F\subseteq Z_{\P}$. Thus, (3) holds.

Conversely, suppose that (3) holds. Let $\P$ be a prime of $N$ dividing $\p_c$ such that $F\subseteq Z_{\P}$. Let $\mathfrak Q=\P\cap Z_{\P}$ and $\q=\P\cap F$. Since $f(\mathfrak Q/\p_c)=1$ and $f(\q/\p_c)$ divides $f(\mathfrak Q/\p_c)$, we have $f(\q/\p_c)=1$. Thus, (2) holds.

The equivalence of (3) and (4) is clear.
\end{proof}

We can now prove the main result for this section.

\begin{thm}\label{main_hit_thm}
Let $M_1,\ldots, M_r$ be representatives of all the conjugacy classes of maximal subgroups of $G$. For $i=1,\ldots, r$ let $F_i$ be the fixed field of $M_i$, and let $f_i(t,x)$ be a monic irreducible polynomial in $k[t][x]$ such that $F_i/k(t)$ is generated by a root of $f_i(t,x)$. Suppose that $c\in k$ satisfies
\begin{equation}\label{hit_excluded_set}
\Delta(c)\cdot\ell(c)\cdot\prod_{i=1}^r\disc f_i(c,x)\ne 0.
\end{equation}
Then the following hold:
\begin{enumerate}
\item If  $\calF(P_c)\ne\calF(P)$, then $G_c\not\cong G$.
\item $G_c\not\cong G$ $\iff$ there is an index $i$ such that $f_i(c,x)$ has a root in $k$.
\end{enumerate}
\end{thm}

\begin{proof}
We prove (1) by contradiction. Thus, suppose that $\calF(P_c)\ne\calF(P)$ and $G_c\cong G$. By Proposition \ref{decomposition_group_prop}, the latter condition implies that the group $G$ acts on the roots of $P$ in the same way that $G_c$ acts on the corresponding roots of $P_c$. Since $\calF(P_c)\ne\calF(P)$, there must be an irreducible factor $f\in k[t,x]$ of $P$ such that $f_c$ is reducible. Note that $G$ acts transitively on the roots of $f$, but since $f_c$ is reducible and separable, $G_c$ does not act transitively on its roots. Thus we have a contradiction, proving (1).

For the proof of (2), suppose that $G_c\not\cong G$ and let $\P$ be a prime of $N$ dividing $\p_c$. By Proposition \ref{decomposition_group_prop}, the group $G_{\P}$ is a proper subgroup of $G$. Replacing $\P$ by a conjugate ideal if necessary, we may therefore assume that $G_{\P}\subseteq M_i$ for some index $i$. By Proposition \ref{decomposition_field_prop} applied to the field $F_i$, this implies that $f_i(c,x)$ has a root in $k$. This proves one direction of (2). The converse follows by a similar argument.
\end{proof}

\subsection{Resolvent polynomial interpretation}
Instead of relying exclusively on the concepts and methods of algebraic number theory, it is also possible to prove a version of Theorem \ref{main_hit_thm} by using only Proposition \ref{decomposition_group_prop} together with the theory of resolvent polynomials, a standard tool in modern methods for computing Galois groups. We recall here the basic definitions and facts needed for our purposes; for further details on this topic as well as proofs of the statements made here, we refer the reader to the articles \cite{stauduhar} and \cite{hulpke}.

Let $n$ be a positive integer and let $G$ and $H$ be subgroups of the symmetric group $S_n$ with $H\le G$. The group $S_n$ acts on the polynomial ring
\[\calR:=\Z[y_1,\ldots,y_n]\] by permuting the variables; for $p\in\calR$ and $\sigma\in S_n$ we will write $p^{\sigma}$ to denote the polynomial $p(y_{\sigma(1)},\ldots, y_{\sigma(n)})$. A polynomial $J\in\calR$ is called a \emph{$G$-relative $H$-invariant} if $H$ is the stabilizer of $J$ in $G$. In that case one associates to the triple $(G,H,J)$ the \emph{resolvent polynomial} $Q=Q_{G,H,J}\in\calR[y]$ given by the formula
\[Q(y_1,\ldots, y_n, y):=\prod_{\sigma\in G//H}(y-J^{\sigma}(y_1,\ldots, y_n)),\]
where $G//H$ denotes a set of representatives for the left cosets $\sigma H$ in $G/H$.

\begin{lem}[Stauduhar \cite{stauduhar}]\label{resolvent_lem}
With notation as above, let $K$ be a field of characteristic $0$ and let $f\in K[x]$ be a separable polynomial of degree $n$. Regard $\Gal(f)$ as a subgroup of $S_n$ by fixing an ordering $\alpha_1,\ldots,\alpha_n$ of the roots of $f(x)$ in $\bar K$, and suppose that $\Gal(f)$ is contained in $G$. If $J$ is any $G$-relative $H$-invariant and $Q$ is the corresponding resolvent, then the polynomial $q(x):=Q(\alpha_1,\ldots,\alpha_n, x)$ has coefficients in $K$. Moreover, if $q(x)$ is separable, then the following are equivalent: 
\begin{enumerate}
\item There exists $\sigma\in G$ such that $\Gal(f)\subseteq\sigma^{-1}H\sigma$.
\item The polynomial $q(x)$ has a root in $K$.
\end{enumerate}
\end{lem}

Returning now to the context of Theorem \ref{main_hit_thm} (2), fix an ordering $x_1,\ldots, x_n$ of the roots of $P(t,x)$ and embed the group $G=\Gal(P)$ into $S_n$ using this ordering. Let $M_1,\ldots, M_r$ be the maximal subgroups defined in Theorem \ref{main_hit_thm}; let $Q_i\in\calR[y]$ be a resolvent polynomial corresponding to a $G$-relative $M_i$-invariant; and let $q_i(t,x)=Q_i(x_1,\ldots, x_n,x)$. By Lemma \ref{resolvent_lem} applied with $K=k(t)$ and $f=P$, the polynomials $q_i(t,x)$ have coefficients in $k(t)$; and since $q_i$ also has coefficients in $\Z[x_1,\ldots, x_n]\subseteq\calO_N$, then $q_i\in A[x]$, where $A$ is the ring defined in \eqref{A_defn}. 

Suppose now that $c\in k$ satisfies $\ell(c)\ne 0$, so that the specialized polynomials $P_c$ and $q_i(c,x)$ are well defined, and suppose that
\[\Delta(c)\cdot\prod_{i=1}^r\disc q_i(c,x)\ne 0.\]
By Proposition \ref{decomposition_group_prop} and its proof, the chosen ordering of the roots of $P$ induces an ordering $\bar x_1,\ldots, \bar x_n$ of the roots of $P_c$, and moreover, as subgroups of $S_n$ we have $G_c\le G$. Note that $q_i(c,x)=Q_i(\bar x_1,\ldots, \bar x_n,x)$.

Applying Lemma \ref{resolvent_lem} once again, but now with $K=k$ and $f=P_c$, we conclude that $G_c$ is contained in a conjugate of $M_i$ if and only if $q_i(c,x)$ has a root in $k$. We thus obtain a variant of Theorem \ref{main_hit_thm} in which the polynomials $f_i$ of the theorem are replaced with the $q_i$ defined here.

\section{Computation of Galois groups over $\Q(t)$}\label{algorithm_section}

We restrict now to the case $k=\Q$. It is clear from Theorem \ref{main_hit_thm} that in order to better understand the exceptional set of a given polynomial $P\in\Q[t][x]$ it is necessary to compute the Galois group $G$, the maximal subgroups of $G$, and their corresponding fixed fields. In this section we discuss a Galois group and fixed field algorithm over $\Q(t)$, the particulars of 
which have not been previously discussed elsewhere.

\subsection{Galois groups of irreducible polynomials over $\Q(t)$}
\label{galois_irred}
The article \cite{fieker-kluners} describes an algorithm to compute Galois groups of 
irreducible polynomials
over $\Q$. As noted in~\cite{fieker-kluners} Section 7.7, this algorithm
can be adjusted to compute Galois groups of polynomials over fields other than the rational
field. For example, \cite{sutherland, suth_thesis} discusses this for polynomials over global
rational and algebraic function fields. An algorithm for computing Galois groups of polynomials over $\Q(t)$
has been implemented in \cite{fieker_implementation} and included in {\sc Magma} V2.15. 

After providing, for ease of reference, a brief summary of the algorithm given 
in \cite{sutherland, suth_thesis} Algorithms 1 and 11, respectively,
we describe here some of the adjustments to this algorithm
that are necessary for these computations over $\Q(t)$.
We address these adjustments using the same headings as \cite{sutherland, suth_thesis}. 
Algorithm \ref{galois_group_algorithm} is a description of the algorithm of~\cite{fieker-kluners}
which uses the descent method of
Stauduhar~\cite{stauduhar} and has no degree restrictions. 

\begin{alg}[Compute the Galois group of a polynomial~(\cite{sutherland} Algorithm 1, \cite{suth_thesis} Algorithm 11)]\mbox{}\label{galois_group_algorithm}

\noindent \textit{Input:} \hspace{1mm} A polynomial $f$ of degree $n$ over $\Q[t]$.\\
\noindent\textit{Output:} The Galois group of $f$.

\begin{enumerate}
\item \label{gal_algo_split}
Compute a splitting field $S_f$ for $f$ over a completion of $F$ and roots of $f$ in $S_f$.

\item \label{gal_algo_start} Find a group $G\subseteq S_n$ which contains $\Gal(f)$
\item \label{outer_loop} While $G$ has maximal subgroups which could contain $\Gal(f)$:
\begin{enumerate}
\item \label{invar} 
For each maximal subgroup $H$ of $G$, compute a 
$G$-relative $H$-invariant polynomial for a representative maximal subgroup
$H$.

\item \label{loop_restart} 
For a cheap maximal subgroup $H$ of $G$ (Stauduhar)

\begin{enumerate}
\item \label{gal_algo_prec} Compute the precision $m$ needed in the roots of $f$
and the roots of $f$ in $S_{f}$ to precision $m$.

\item \label{short_cosets} For the representatives $\tau \in G//H$ of the right cosets of $H$ in $G$:
evaluate $I_H^{\tau}$ at the roots of $f$.
Decide whether this is the image of an element of $F$ in $S_{f}$.
If so %the resolvent has a root in $F$ %and if it is a single root
$\Gal(f) \subseteq \tau H\tau^{-1}$
and restart the loop~(\ref{outer_loop}) with $G = \tau H\tau^{-1}$.

\end{enumerate}
\end{enumerate}
\item Return $G$.
\end{enumerate}
\end{alg}

We now discuss how some of the steps of the above algorithm can be carried out in the case where $f$ is irreducible. 

\label{irred}
\begin{description}[labelindent=0mm,leftmargin=0cm,itemsep=2mm]
\item[Choosing a good prime] (Step~\ref{gal_algo_split}). A good prime is necessary for computing a completion of $\Q(t)$ and a
splitting field over this completion. The image of $f$ must be squarefree over the residue field at $P$.  Instead of the completion being a $p$-adic field (completion of the rationals) or a series field over the field
of constants (completion of a global rational function field) we complete
in two directions and compute a completion as a series field over a $p$-adic 
field. For this we need two primes, an integer prime for computing a $p$-adic 
field and a polynomial prime to compute a series field over this $p$-adic 
field.

The choice of a good polynomial prime can be undertaken in the same way as for 
the global function fields; see ~\cite{sutherland} Section 3.1 or ~\cite{suth_thesis} Section 8.1. 
In contrast to the global case we consider only $n$ primes. 
Let $r_P$ be
the degree of $P$, $d_P$ the LCM of the degrees of the factors of the image 
of $f$ mapped over $\Q[x]/P$, and let $l_{f,P}$ be the number of factors of the image
of $f$ mapped over $\Q[x]/P$.
Similar to the global case 
we choose $P$ with the smallest $r_P d_P l_{f,P}^{1.5} > n/4$ if such occurs for a prime we have considered; otherwise a prime we
have considered with largest $r_P d_P l_{f,P}^{1.5} \le n/4$.

To choose a good integer
prime for the $p$-adic part of the completion we construct the number field
$K = \Q[x]/P(x)$, where $P$ is the prime polynomial chosen. We map $f$ to a 
polynomial $f_K$ over $K$ and compute a prime $p$
which is good for the computation of the Galois group of $f_K$, a polynomial
over a number field. Lemma 2.16 of \cite{geissler_klueners} contains some necessary
conditions such primes must satisfy. As discussed in~\cite{sutherland, suth_thesis} Section 3.1 and 8.1, respectively, we choose the prime $p$
so that the
extension of the $p$-adic field is not of too large degree to be expensive to 
work in nor of too small
degree that computations will require excessive precision, 

\item[Computing roots]
(Step~\ref{gal_algo_split})
We construct the field $\Q_p(\alpha)(\!(z)\!)$ which will contain all the roots of 
the image of $f$. This splitting field is a combination of the
extension $\Q_p(\alpha)$ of a $p$-adic field used 
when computing Galois groups of polynomials over $\Q$ or a 
number field and the extension $\F_{(q^{r_P})^{d_P}}(\!(z)\!)$ used when computing
Galois groups of polynomials over $\F_q(t)$ or an extension thereof.
The local field $\Q_p(\alpha)$ contains the roots of $f_K$ mapped over
the completion of $K$ at $p$.

We take $z$ as the image of $P$ in $K(\!(z)\!)$ and use the map 
$h : \Q(t) \rightarrow K(\!(z)\!)$ given by the completion mapping at $P$ into 
$K(\!(z)\!)$, and then combine with the mappings 
$K \rightarrow K_p \rightarrow K_p(\beta) = \Q_p(\alpha)$.
To compute the roots of $f$ we first compute the roots of $f_K$ in 
$\Q_p(\alpha)$ to the required $p$-adic precision
and Hensel lift to the required $P$-adic precision. 

\item[A starting group]
(Step~\ref{gal_algo_start})
Section 3.3 of \cite{sutherland, suth_thesis} applies also for polynomials over $\Q(t)$.
Subfields over $\Q(t)$ can be computing using~\cite{klueners_subfield}.

\item[Invariants]
(Step~\ref{invar})
The invariants and Tschirnhausen transformations
presented in \cite{fieker-kluners} are sufficient as all 
rings involved here have characteristic 0.

\item[Mapping back to the function field]
(Step~\ref{short_cosets})
Given a series $g \in \Q_p(\alpha)(\!(z)\!)$ we check whether the coefficients of
$g$ map back to elements of $K$. To the resulting series now in $K(\!(z)\!)$ we
apply the homomorphism which maps $z$ to $P$ and the coefficients to 
polynomials over $\Q$ using a 
homomorphism mapping the generator of $K$ to a root of $P$ in $\Q(t)/P^r$, 
where $r$ is the $P$-adic precision of $g$.
Lastly we take the remainder of this resulting polynomial mod $P^r$.

\item[Determining a descent]
(Step~\ref{short_cosets})
Most of the discussion in~\cite{sutherland} Section 3.8 
and~\cite{suth_thesis} Section 8.8 applies here, including bounding the 
degree of the evaluation of an invariant at the roots of $f$.
However, just as we required two primes, we also 
require two bounds -- one on the polynomial degree of an evaluation, and one on the size
of the coefficients of that polynomial. The minimum infinite valuation, computed in the same way as for global fields,
can be used to compute a precision for a series $M$ over the 
integers (computed using complex roots) which is a bound for the complex 
size of the integral coefficients.
This is used to compute 
a bound $B$ on the $T_2$ norm as $\deg(P)$ times the square of 
a bound on the evaluation of the invariant at a transformed root of
size $T(M)$, where $T$ is a Tschirnhausen transformation.
The absolute precision of $B$ times $\deg(P)$ is used to bound the degree
of the evaluation of an invariant mapped back to $\Q(t)$.
The maximum coefficient of $B$ is then used to bound the coefficients of this mapped
evaluation.

\item[Precision]
Since we have two completions in our splitting field construction, we require 
both a $p$-adic precision and a series precision.
These are computed from the bound computed above.
The series precision is taken to be the absolute precision of the series bound
and the $p$-adic precision is computed from the largest coefficient of the
series bound using Proposition 3.12 in~\cite{Belabas2004641}. 
\end{description}

\subsection{Galois groups of reducible polynomials over $\Q(t)$}\label{reducible_galois_section}
Section 7.6 of \cite{fieker-kluners} mentions that their algorithm can be used to compute Galois groups of reducible polynomials. Indeed it can be, but as with the different coefficient rings some adjustments
must be made. For reducible polynomials these adjustments are most specifically
in the computation of a starting group and the handling of multiple roots and linear factors.
These adjustments necessary to efficiently
compute Galois groups of reducible polynomials over global rational and algebraic function
fields are discussed in \cite{sutherland, suth_thesis} and included in {\sc Magma} V2.18. Here we
describe the similar necessary adjustments, included in {\sc Magma} V2.23, to use the algorithm of \cite{fieker-kluners} to efficiently compute Galois groups of reducible polynomials over $\Q(t)$.
We use Algorithm 2 of \cite{suth_thesis} (\cite{sutherland} Algorithm 12) and address these adjustments using the same headings used there. This algorithm uses the product of some of the Galois groups of the factors of $f$ to gain a smaller starting group
in Step \ref{gal_algo_start} of Algorithm \ref{galois_group_algorithm},
and also does some post processing with linear and multiple factors so that the descent steps are minimized.

\begin{description}[labelindent=0mm,leftmargin=0cm,itemsep=2mm]
\item[Choosing a good prime] 
The same polynomial prime and same integer prime must be used
to compute the Galois groups of each of the factors of $f$. Otherwise 
Section 4.1.1 of \cite{sutherland} (Section 9.2.1 of \cite{suth_thesis}) holds also for polynomials over $\Q(t)$.
\item[Computing roots in the splitting field over the completion]
The local field $\Q_p(\alpha)$ must be computed such that it contains 
the roots of $(f_i)_K$ over $\Q_p$ for all factors $f_i$ of $f$.
The field $\Q_p(\alpha)(\!(z)\!)$ can then be used as a splitting field.

\item[Check disjointedness of splitting fields] 
By checking whether the splitting fields of the factors are disjoint we can
reduce the size of the descent necessary by starting with a smaller group.
Since~\cite{stichtenoth} Theorem III.6.3 and Corollary III.5.8 hold when the 
constant field (in this case $\Q$) is a perfect field, the discussion
of~\cite[\S4.1.3]{sutherland} and \cite[\S9.2.3]{suth_thesis} holds when computing
Galois groups of reducible polynomials over $\Q(t)$ as well as over $\F_q(t)$.

\item[Invariants]
The invariants presented in \cite{fieker-kluners} are sufficient here as all 
the rings involved have characteristic zero.
\item[Determination] 
The computation of the precision
necessary is as in Section~\ref{galois_irred}, using the minimum infinite valuation
(negative of the maximum degree)
of all scaled roots of $f$. 
\item[Multiple and linear factors]
Section 4.1.6 of \cite{sutherland} (or Section 9.2.6 of \cite{suth_thesis}) holds also for polynomials
over $\Q(t)$.
\end{description}

\subsection{Computing a fixed field of a subgroup of a Galois group}\label{fixed_field_section}
The procedure needed for this computation, which was
implemented in {\sc Magma} by Fieker and Kl\"uners,
is independent of the coefficient ring of the polynomial. We summarize an
algorithm below.
Though the details differ between coefficient
rings, the necessary adjustments are already addressed in the various descriptions of
the Galois group algorithm given in~\cite{fieker-kluners, sutherland, suth_thesis} and above. This algorithm applies to both reducible and irreducible polynomials.

\begin{alg}[Compute a fixed field of a subgroup of a Galois group]\label{fixed_field_algorithm}
Given a subgroup $U$ of a Galois group $G$ of a polynomial $f$ of degree $n$, and given the data used to compute $G$ from $f$, compute a defining polynomial for the fixed field of $U$.
\begin{enumerate}
\item Compute a $G$-relative $U$-invariant polynomial $I$
and the right transversal $G//U$.
\item Compute a Tschirnhausen transformation $T$ such that 
$$\#\{I^{\tau}(T(r_1), \ldots, T(r_n)) : \tau \in G//U\} = \#G//U.$$
using roots $\{r_i\}_{i=1}^n$ of $f$ to some low precision in the 
splitting field used for the Galois group computation.
\item Compute a bound $B$ on the evaluation of the invariant $I$ at the roots of
$f$
and the roots $\{r_i\}_{i=1}^n$ of $f$ to a precision that allows
the bound $B$ to be used.
\item Compute the monic polynomial $g$ with roots 
$$\{I^{\tau}(T(r_1), \ldots, T(r_n)) : \tau \in G//U\}.$$
\item Map the coefficients of $g$ back to the coefficient ring of $f$, and return the 
resulting polynomial, which is a defining polynomial for the fixed field of $U$.
\end{enumerate}
\end{alg}

\begin{rem}The polynomial returned by Algorithm \ref{fixed_field_algorithm} will be of 
degree $\#(G//U)$; this can cause difficulties in practice when $G//U$ is large. An example where this type of issue occurs is discussed in \S\ref{large_example_section}.
\end{rem}

\subsection{Proof of Galois groups of polynomials over $\Q(t)$}

The correctness of Galois groups computed using lower precision than necessary, such as when $[G:H]$ is substituted with a smaller value in the precision computation, can be proved
using absolute resolvents as 
in~\cite{geissler_klueners} Algorithm 5.1 
and~\cite{fieker-kluners} Section 7.4. We consider here the adjustments to these algorithms needed over $\Q(t)$.

Suppose we know that $H \le \Gal(f) \le G$.
Algorithm 5.1 of~\cite{geissler_klueners} determines whether
$\Gal(f) \not= G$ or $\Gal(f) \not= H$ 
by computing a resultant polynomial $R$ and two factors, $f_1$ and $f_2$,
of $R$ to precision 1 based on an $H$-orbit which is not a $G$-orbit. 
This factorization is lifted to $F_1 F_2$ with
precision $k$, where $k$ is computed
from a bound $M$ on the coefficients of the factors of $R$ as in the computation of the Galois group. If $F_1$ corresponds
to a true factor of $R$, then $\Gal(f) \not= G$; otherwise $\Gal(f) \not= H$.
Letting $\{\alpha_i\}$ be the roots of $f$, we use these bounds to bound the coefficients of 
$$R(y) = \prod_{\tau} (y - I^{\tau}(\alpha_1, \ldots, \alpha_n)),$$
where $I$ is $H$-invariant, and its factors by the quantity $$\max_{1\le i\le n} \left\{\binom{\deg(R)}{i}\right\} B^{\deg(R)},$$ where
$B$ is a bound on $I^{\tau}(\alpha_1, \ldots, \alpha_n)$ obtained
as in the ``Determining a descent" step of Section~\ref{galois_irred}.

\section{Examples}\label{examples_section}
Having developed the theoretical and algorithmic material that form the core of this article, we proceed to apply our results to study the exceptional sets of three sample polynomials. The following algorithm will be our main tool.

\begin{alg}\label{hit_data_algorithm}\mbox{}

\smallskip
\noindent \textit{Input:} \hspace{1mm} A separable polynomial $P\in \Q[t][x]$.\\
\noindent\textit{Output:} A finite set $D\subset \Q$ and a finite set $S\subset \Q[t][x]$.

\begin{enumerate}
\item Create empty sets $D$ and $S$.
\item Include in $D$ all the rational roots of the discriminant of $P$.
\item Include in $D$ all the rational roots of the leading coefficient of $P$.
\item Compute the group $G=\Gal(P)$. More precisely, find a permutation representation of $G$ induced by a labeling of the roots of $P$.
\item Find subgroups $M_1,\ldots, M_r$ representing all the conjugacy classes of maximal subgroups of $G$.
\item For $M\in\{M_1,\ldots, M_r\}$: 
\begin{enumerate}
\item Find a monic irreducible polynomial $f\in \Q[t][x]$ such that the fixed field of $M$ is generated by a root of $f$.
\item Include $f$ in the set $S$.
\item Include in $D$ all the rational roots of the discriminant of $f$.
\end{enumerate}
\item Return the sets $D$ and $S$.
\end{enumerate}
\end{alg}

For Step 4 we use Algorithm \ref{galois_group_algorithm} when $P$ is irreducible and the adjustments discussed in \S\ref{reducible_galois_section} when $P$ is reducible. For Step 5 we use an algorithm of Cannon and Holt \cite{cannon-holt}. For Step 6(a) we use Algorithm \ref{fixed_field_algorithm} and adjust the polynomial returned so that it has coefficients in $\Q[t]$ and defines the same extension. All of our computations were done using \textsc{Magma} V2.23, which includes implementations of these algorithms. The intrinsic function \texttt{HilbertIrreducibilityCurves} in \textsc{Magma} V2.24 is an implementation of Algorithm \ref{hit_data_algorithm}.

For later reference we record the following consequence of Theorem \ref{main_hit_thm}.

\begin{prop}\label{algorithm_output_prop}
Let $P\in \Q[t][x]$ be a separable polynomial with Galois group $G$, and let $D$ and $S$ form the output of Algorithm \ref{hit_data_algorithm} with input $P$. Then for all $c\in \Q\setminus D$ we have:
\begin{enumerate}
\item If $\calF(P_c)\ne\calF(P)$, then $G_c\not\cong G$.
\item $G_c\not\cong G\iff$ there exists $f\in S$ such that $f(c,x)$ has a rational root.
\end{enumerate} 
\end{prop}

\subsection{A finite exceptional set}\label{fermat_section} In our first example we consider the polynomial $P(t,x)=x^6+t^6-1$.
As follows from the case $n=3$ of Fermat's Last Theorem, the specialized polynomial $P_c$ has a rational root if and only if $c\in\{0,\pm 1\}$. We will prove the following stronger result.

\begin{prop}\label{fermat_specializations_prop}
Let $c\in\Q$. Then $P_c$ is reducible if and only if $c\in\{0,\pm1\}$.
\end{prop}
\begin{proof}
Suppose that $P_c$ is reducible and that $c\notin\{0,\pm1\}$. Applying Algorithm \ref{hit_data_algorithm} to the polynomial $P$ we obtain the set $\{-1,1\}$ and the polynomials
\begin{align*}
F_1(t,x) &= x^2 - 2^8\cdot 3^5(t^6-1)^5,\\
F_2(t,x) &= x^2 + 64\cdot 27(t^6-1)^2,\\
F_3(t,x) &= x^2 + 6x + 9t^6,\\
F_4(t,x) &= x^3 + 12x^2 + 48x  - 8t^6 + 72.
\end{align*}

By Proposition \ref{algorithm_output_prop}, at least one of the polynomials $F_i(c,x)$ must have a rational root; we accordingly divide the proof into four cases.

\underline{\textit{Case 1:}} There exists $r\in\Q$ such that $F_1(c,r)=0$. Defining $u=c^2$ and
\[v=r/\left(2^4\cdot 3^2\cdot(c^6 - 1)^2\right),\]
the equation $F_1(c,r)=0$ implies that $v^2=3(u^3-1).$ This equation defines the elliptic curve with Cremona label 36a3, which has rank 0, and its only affine rational point is $(1, 0)$. It follows that $u=1$ and thus $c=\pm 1$, which is a contradiction. Hence this case cannot occur.

\underline{\textit{Case 2:}} There exists $r\in\Q$ such that $F_2(c,r)=0$. Letting $u=8\cdot 3\cdot(c^6-1)$,
we have $u\ne 0$ and $r^2+3u^2=0$, which is clearly impossible. Thus we have a contradiction.

\underline{\textit{Case 3:}} There exists $r\in\Q$ such that $F_3(c,r)=0$. Letting $v=(r+3)/3$ and $u=-c^2$,
the equation $F_3(c,r)=0$ implies that $v^2=u^3+1$. This equation defines the elliptic curve with Cremona label 36a1, which has rank 0 and a torsion subgroup of order 6; its only affine rational points are $(0,\pm 1)$, $(2,\pm 3)$, and $(-1,0)$. Since $u<0$, we must have $u=-1$ and therefore $c^2=1$, which is a contradiction.

\underline{\textit{Case 4:}} There exists $r\in\Q$ such that $F_4(c,r)=0$. Letting $y=2c^3$, the equation $F_4(c,r)=0$ implies that $2y^2=r^3+12r^2+48r+72$. This equation defines the elliptic curve with Cremona label 36a1, the same curve that appeared in the previous case. Using the above model of the curve, the affine rational points are $(0,\pm 6)$, $(-4,\pm 2)$, and $(-6,0)$.
It follows that $y=\pm 6, \pm 2$, or 0, which implies that $c^3=\pm 3$, $c=\pm 1$, or $c=0$, all of which yield a contradiction.

Since every case has led to a contradiction, we conclude that $c\in\{0,\pm 1\}$.
\end{proof}

\subsection{An infinite family of exceptional factorizations}\label{S4_example_section} Let $P(t,x)=x^6 - 4x^2 - t^2$, which is an irreducible polynomial with Galois group isomorphic to the symmetric group $S_4$. In this example we will determine precisely for which rational numbers $c$ the specialization $P_c$ is reducible, and how $P_c$ factors in that case.

\begin{prop}\label{S4_example_exceptional_set}
Let $c$ be a nonzero rational number such that $P_c$ is reducible. Then $c$ has the form $c=(v^4+16)/(8v)$ for some $v\in\Q$.
\end{prop}
\begin{proof}
Applying Algorithm \ref{hit_data_algorithm} to $P$ we obtain the set $\{0\}$ and the polynomials
\begin{align*}
F_1(t,x) &= x^4 - 8tx + 16,\\
F_2(t,x) &= x^2 + 1728t^4 - 16384,\\
F_3(t,x) &= x^3 + 24x^2 + 176x - 8t^2 + 384.
\end{align*}

By Proposition \ref{algorithm_output_prop}, one of the polynomials $F_i(c,x)$ must have a rational root. We will show that $i$ cannot be 2 or 3, from which the proposition follows easily.

Suppose that $F_2(c,x)=0$ for some $x\in\Q$. Letting
\[X=\frac{x+128}{4c^2}\;\;\text{and}\;\;Y=\frac{2x+256}{c^3}\]

we obtain $Y^2=X^3+108X$. This equation defines an elliptic curve with exactly two rational points, namely the point at infinity and the point $(0,0)$. Hence we must have $X=Y=0$, which implies that $x=-128$. However, the equation $F_2(c,-128)=0$ implies that $c=0$, which is a contradiction.

By a similar argument one can show that the only rational solutions to the equation $F_3(t,x)=0$ are $(0,-4)$, $(0,-8)$, and $(0,-12)$; hence $F_3(c,x)=0$ is impossible for $x\in\Q$ since $c\ne 0$.
\end{proof}

\begin{lem}\label{S4_example_genus3_lem}
Let $\calC$ and $\calD$ be the curves in $\A^2=\Spec\Q[v,x]$ defined by the equations $8vx^3 - 8v^2x^2 + 4v^3x - v^4 - 16=0$ and $8vx^3 + 8v^2x^2 + 4v^3x + v^4 + 16=0$, respectively. Then $\calC(\Q)=\calD(\Q)=\emptyset$.
\end{lem}
\begin{proof}
The curves $\calC$ and $\calD$ are both non-hyperelliptic curves of genus 3, but they admit a map to an elliptic curve of rank 0; this allows us to determine their rational points. We give the proof only for $\calC$, since the argument is very similar for $\calD$.

There is a map from $\calC$ to the elliptic curve $E:Y^2 =X^3 + X^2 + X$ given by $(v,x)\mapsto\left((2x - v)/v, (-4)/v^2\right)$. The curve $E$ has rank 0, and its only rational points are $\infty$ and $(0,0)$. Any rational point on $\calC$ must necessarily have $v\ne 0$, and will therefore map to the point $(0,0)$ on $E$. However, this is impossible since $-4/v^2\ne 0$. Hence $\calC$ has no rational point.
\end{proof}

\begin{prop}\label{S4_example_factorization_prop}
Suppose that $c$ is of the form $c=(v^4+16)/(8v)$ for some rational number $v$. Then $P_c$ factors as
\[64v^2\cdot P_c=(8vx^3 - 8v^2x^2 + 4v^3x - v^4 - 16)(8vx^3 + 8v^2x^2 + 4v^3x + v^4 + 16).\]
Moreover, both cubic factors of $P_c$ are irreducible.
\end{prop}
\begin{proof}
Substituting $c=(v^4+16)/(8v)$ in the polynomial $P(c,x)$ and factoring, we obtain the above factorization. Lemma \ref{S4_example_genus3_lem} implies that neither factor of $P_c$ can have a rational root, and therefore both factors are irreducible.
\end{proof}

\subsection{An infinite family of exceptional Galois groups}\label{serre_section} In \cite[\S4.5]{serre_topics} Serre shows that for even values of $n$, the polynomial 
\[P_n(t,x)=(n-1)x^n - nx^{n-1}+1+(-1)^{n/2}(n-1)t^2\]
has the alternating group $A_n$ as its Galois group. By HIT, most specializations $P_n(c,x)$ will have Galois group $A_n$ as well. In the case $n=4$ we obtain the polynomial
\[P(t,x)=3x^4-4x^3+1+3t^2\]
with Galois group $A_4$. In this example we will determine precisely for which rational numbers $c$ the Galois group $G_c$ is different from $A_4$, and which groups $G_c$ arise for such numbers $c$. Our main results are Propositions \ref{serre_specializations_prop} and \ref{serre_ex_subgroups}.

\begin{lem}\label{serre_ex_F1_lem}
Let $F_1(t,x) = x^3 + 48x^2 + (336-1296t^2)x - 10368t^2 + 640$ and let $c\in\Q^{\ast}$. Then the polynomial $F_1(c,x)$ has a rational root if and only if $c$ has the form
\begin{equation}\label{serre_ex_eq}
c=\frac{v^3 - 9v}{9(1 - v^2)}
\end{equation}
for some rational number $v$.
\end{lem}
\begin{proof}
Let $C$ be the plane curve defined by the equation $F_1(t,x)=0$. The curve $C$ is parametrizable; indeed, the rational maps
\[\phi:C\dashedrightarrow\A^1=\Spec\Q[z]\;\;\text{and}\;\; \psi:\A^1\dashedrightarrow C\]
given by
\[\psi(z)=\left(\frac{z^3-9z}{9(1-z^2)},\frac{8(z^2-5)}{1-z^2}\right)\;\;\text{and}\;\;\phi(t,x)=\frac{x^2 - 1296t^2 + 44x + 160}{144t}\]
are easily seen to be inverses.

Suppose that $c$ is of the form \eqref{serre_ex_eq}. We may then define
\[r=\frac{8(v^2-5)}{1-v^2},\]
so that $\psi(v)=(c,r)$ is a rational point on $C$. Hence, the polynomial $F_1(c,x)$ has a rational root (namely $r$).

Conversely, suppose that $F_1(c,x)$ has a rational root, say $r$. Since $c\ne 0$, the map $\phi$ is defined at the point $(c,r)\in C(\Q)$. Thus, we may define $v=\phi(c,r)$. We claim that $v\ne\pm 1$. A straightforward calculation shows that the rational points on the pullback of $\pm 1$ under $\phi$ are $(0,-40)$ and $(0,-4)$. Since $c\ne 0$, the point $(c,r)$ is different from these two points. Hence $v=\phi(c,r)\ne\pm 1$, as claimed. The map $\psi$ is therefore defined at $v$, so $(c,r)=\psi(v)$. In particular, $c$ is of the form \eqref{serre_ex_eq}.
\end{proof}

\begin{lem}\label{serre_ex_F2_lem}
Let $F_2(t,x) = x^4 + 4x^3 + 81t^2 + 27$ and let $c\in\Q^{\ast}$. Then the polynomial $F_2(c,x)$ has no rational root.
\end{lem}
\begin{proof}
Suppose that $r\in\Q$ is such that $F_2(c,r)=0$. Since $c\ne 0$, we must have $r\ne-3$. Defining $y=9c/(r+3)$,
the equation $F_2(c,r)=0$ implies that
\[y^2+(r-1)^2+2=0,\]
which is clearly impossible for $y,r\in\Q$. This contradiction proves the lemma.
\end{proof}

\begin{prop}\label{serre_specializations_prop}
Let $c\in\Q$ and let $G_c$ be the Galois group of $P_c$. Then
\[G_c\not\cong A_4\iff c=\frac{v^3 - 9v}{9(1 - v^2)}\;\;\text{for some}\;v\in\Q.\]
\end{prop}
\begin{proof}
For $c=0$ the proposition holds because both statements in the above equivalence are true. Indeed, we have
\[P_0=3x^4-4x^3+1=(x-1)^2(3x^2+2x+1),\]
so $G_0$ has order 2. 

Suppose now that $c\ne 0$. Applying Algorithm \ref{hit_data_algorithm} to the polynomial $P$ we obtain the set $\{0\}$ and the polynomials
\begin{align*}
F_1(t,x) &= x^3 + 48x^2 + (336-1296t^2)x - 10368t^2 + 640,\\
F_2(t,x) &= x^4 + 4x^3 + 81t^2 + 27.
\end{align*}

By Proposition \ref{algorithm_output_prop} and Lemmas \ref{serre_ex_F1_lem} and \ref{serre_ex_F2_lem}, we have the following:
\begin{align*}
G_c\not\cong A_4 &\iff F_1(c,x)\cdot F_2(c,x)\;\text{has a rational root}\\
&\iff F_1(c,x)\;\text{has a rational root}\\
&\iff c=\frac{v^3 - 9v}{9(1 - v^2)}\;\;\text{for some}\;v\in\Q.\qedhere
\end{align*}
\end{proof}

Continuing with this example, we turn now to the question of which subgroups of $A_4$ arise as groups $G_c$ for some $c\in\Q$. Let $N$ be a splitting field for $P$ over $\Q(t)$. Let $G=\Gal(N/\Q(t))$ be the Galois group of $P$, and fix an isomorphism $G\cong A_4$. Since
\[\disc P=2^8\cdot3^4\cdot t^2(3t^2+1)^2,\]
Proposition \ref{decomposition_group_prop} implies that for $c\ne 0$, the group $G_c$ is isomorphic to a subgroup of $A_4$. The same holds true for $c=0$ since $G_0$ has order 2 and $A_4$ has a subgroup of order 2. Proposition \ref{serre_specializations_prop} thus leads naturally to the following question: when $G_c$ is not isomorphic to $A_4$, which subgroup of $A_4$ is it isomorphic to? The methods of \S\S\ref{hit_section}-\ref{algorithm_section} provide a way to answer this.

Up to conjugacy, $A_4$ has exactly three nontrivial proper subgroups:
\[A=\langle(1,2)(3,4), (1,3)(2,4)\rangle,\; B=\langle(1,2,3)\rangle,\; \text{and} \;C=\langle(1,2)(3,4)\rangle.\]
We will henceforth identify $A$, $B$, and $C$ with the subgroups of $G$ that they correspond to under the isomorphism $G\cong A_4$.

Let $F_A, F_B$, and $F_C$ be the fixed fields of $A, B$, and $C$, respectively. Using \textsc{Magma} we find that $F_A$ and $F_B$ are generated over $\Q(t)$ by the polynomials $F_1(t,x)$ and $F_2(t,x)$ defined in Lemmas \ref{serre_ex_F1_lem} and \ref{serre_ex_F2_lem}. For $F_C$ we obtain the polynomial
\begin{multline*}
F_3(t,x)=x^6 + 12x^5 + 48x^4 + 64x^3 \\ 
- (324t^2 + 108)x^2 - (1296t^2 + 432)x - 1296t^2 - 432.
\end{multline*}

\begin{lem}\label{serre_ex_nonsubgroups}
Let $c\in\Q^{\ast}$, and let $\P$ be a prime of $N$ lying over the prime $(t-c)\subset \Q[t]$. Then $G_{\P}\not\subseteq B$.
\end{lem}
\begin{proof}
Suppose, by contradiction, that $G_{\P}\subseteq B$. We have
\[\disc F_2=2^8\cdot 3^{10}\cdot t^2(3t^2+1)^2,\]
so $\disc F_2(c,x)\ne 0$. By Proposition \ref{decomposition_field_prop} applied to the field $F_B$, the polynomial $F_2(c,x)$ has a rational root. However, this contradicts Lemma \ref{serre_ex_F2_lem}.
\end{proof}

It follows from the above lemma (and Proposition \ref{decomposition_group_prop}) that if $c\in\Q$ is such that $G_c\not\cong A_4$, then $G_c$ must be isomorphic to either $A$ or $C$. We now determine precisely when each case occurs.

\begin{lem}\label{serre_ex_F3_lem}
Let $v\in\Q\setminus\{\pm1\}$ and let $c=(v^3-9v)/(9(1-v^2))$.
Then the polynomial $F_3(c,x)$ has a rational root if and only if $v$ has one of the forms
\[1+2w^2,\;-1-2w^2,\; or\; \frac{2w}{1+w^2}\]
for some $w\in\Q$.
\end{lem}
\begin{proof}
Substituting $c=(v^3-9v)/(9(1-v^2))$ in the polynomial $F_3(c,x)$ we find that there is a factorization
\[(v^2-1)^2\cdot F_3(c,x)=f(v,x)\cdot g(v,x)\cdot h(v,x),\]
where $f,g,h\in\Q[t,x]$ are defined by
\begin{align*}
f(t,x) &=(t - 1)x^2 + (4t - 4)x - 2t^2 - 6,\\
g(t,x) &=(t + 1)x^2 + (4t + 4)x + 2t^2 + 6,\\
h(t,x) &=(t^2 - 1)x^2 + (4t^2 - 4)x + 4t^2 + 12.
\end{align*}
It follows that $F_3(c,x)$ has a rational root if and only if at least one of the discriminants of $f(v,x)$, $g(v,x)$, or $h(v,x)$ is a square. Now, it is a straightforward calculation to verify that
\[\disc f(v,x)=8(v - 1)(v + 1)^2\]
is a square if and only if $v=1+2w^2$ for some $w\in\Q$; that
\[\disc g(v,x)=-8(v + 1)(v - 1)^2\] is a square if and only if $v=-1-2w^2$ for some $w\in\Q$; and that
\[\disc h(v,x)=-64(v^2 - 1)\] is a square if and only if $v=(2w)/(1+w^2)$ for some $w\in\Q$. Therefore, $F_3(c,x)$ has a rational root if and only if $v$ has one of these forms.
\end{proof}
We can now give a complete characterization of the groups $G_c$ that are not isomorphic to $A_4$. 
\begin{prop}\label{serre_ex_subgroups}
Suppose that $c\in\Q$ satisfies $G_c\not\cong A_4$, so that
\[c=\frac{v^3 - 9v}{9(1 - v^2)}\;\;\text{for some}\;v\in\Q.\]
If $v$ has one of the forms
\begin{equation}\label{serre_ex_v_forms}
1+2w^2,\;-1-2w^2,\; or\; \frac{2w}{1+w^2},
\end{equation}
then $G_c\cong C$. Otherwise $G_c\cong A$.
\end{prop}
\begin{proof}
For $c=0$ the result is easily verified; thus, we assume that $c\ne 0$. We then have $\disc F_3(c,x)=2^{28}\cdot 3^{20}\cdot c^4(3c^2+1)^4\ne 0$.

Suppose that $v$ has one of the forms \eqref{serre_ex_v_forms}. By Lemma \ref{serre_ex_F3_lem}, the polynomial $F_3(c,x)$ has a rational root. By Proposition \ref{decomposition_field_prop} this implies that there exists a prime $\P$ of $N$ lying over $(t-c)$ such that $G_{\P}\subseteq C$. The group $G_{\P}$ is nontrivial by Lemma \ref{serre_ex_nonsubgroups}, so $G_{\P}=C$. Hence, by Proposition \ref{decomposition_group_prop}, $G_c$ is isomorphic to $C$.

Suppose now that $v$ cannot be written in any of the forms \eqref{serre_ex_v_forms}. By Lemma \ref{serre_ex_F3_lem} and Lemma \ref{serre_ex_F1_lem}, the polynomial $F_3(c,x)$ does not have a rational root but $F_1(c,x)$ does. Hence, by Proposition \ref{decomposition_field_prop}, there exists a prime $\P$ of $N$ lying over $(t-c)$ such that $G_{\P}\subseteq A$ but $G_{\P}\not\subseteq C$. Hence $G_{\P}=A$ and, by Proposition \ref{decomposition_group_prop}, $G_c\cong A$.
\end{proof}

\subsection{A reducible polynomial}\label{reducible_example_section} We end \S\ref{examples_section} by mentioning two sample computations with polynomials $P(t,x)$ of larger degrees. Let $\phi(x)=x^2+t\in\Q(t)[x]$ and let $n$ be a positive integer. The \emph{$n$-th dynatomic polynomial} of $\phi$ is defined by the equation
\[\Phi_n(t,x)=\prod_{d|n}\left(\phi^d(x)-x\right)^{\mu(n/d)},\]
where $\mu$ is the M\"{o}bius function and $\phi^d$ denotes the $d$-fold composition of $\phi$ with itself. The significance of the polynomial $\Phi_n$ is that its roots are precisely the elements of $\overline{\Q(t)}$ having period $n$ under iteration of $\phi(x)$. It is an important unsolved problem in the field of arithmetic dynamics to determine the exceptional set of $\Phi_n$ for every $n$.

\bigskip
Consider the problem of determining the exceptional set of the polynomial $\Phi_4$, which has degree 12. Using Algorithm~\ref{hit_data_algorithm} we find that every rational number $c$ of the form $c = (4-3v-v^3)/4 v$ with $v\in\Q^\ast$ lies in $\calE(\Phi_4)$. To determine what the Galois group of the specialized polynomial $\Phi_4(c,x)$ is when $c$ has this form, we compute the Galois group, $H$,
of the polynomial
\[P(v, x):= \Phi_4((4-3v-v^3)/4 v, x)\in\Q(v)[x],\]
which is a product of a degree-8 and a degree-4 polynomial. An application of Algorithm~\ref{hit_data_algorithm} to $P(v,x)$ reveals that when $v$ does not have
the form $v=(1+4s^3-s^6)/(4s^2(s^2-1))$, the Galois group of $P(v,x)$ is equal to $H$, and when $v$ does have this form, the Galois group is a proper subgroup of $H$. Note that these calculations require the
ability to compute Galois groups of reducible polynomials over function fields.

\subsection{A polynomial of degree 30}\label{large_example_section}
Consider now the polynomial $P=\Phi_5$. The Galois group $G$ in this case is the wreath product $(\Z/5\Z)\wr S_6$, which has order 11,250,000. While the computation of $G$ and its maximal subgroups did not present any problems, we were unable to fully carry out Algorithm~\ref{hit_data_algorithm} for this polynomial. The unique obstacle was the computation of the fixed field of one maximal subgroup of $G$ that has index 3125. For the other maximal subgroups (all of which have index at most 15) we did successfully compute the corresponding fixed field.

\section{An application to arithmetic dynamics}\label{dynamics_section}

In \cite{manes} Manes studies the possible periods of rational numbers under iteration of maps of the form
\[\phi_{k,b}(z)=kz+b/z,\;\;\text{where}\;\;k\in\Q\setminus\{0,-1/2\}\;\;\text{and}\;\;b\in\Q^{\ast}.\]
Manes conjectures that no such map can have a rational point of period greater than four, and proves the following finiteness result for points of period five: there exists a finite set $S\subset\Q$ such that if $k\notin S$ then, for any $b\in\Q^{\ast}$, the map $\phi_{k,b}$ has no point of period five in $\Q$. Our goal in this section is to prove a stronger statement.

\begin{prop}\label{manes_specializations_prop}
There exists a finite set $S\subset\Q$ such that if $k\in\Q\setminus S$, then the following holds for every $b\in\Q^{\ast}$: for at least one third of all prime numbers $p$ (in the sense of Dirichlet density), the map $\phi_{k,b}$ has no point of period five in $\Q_p$.
\end{prop}

To prove the above proposition we will need to determine the exceptional set of polynomial $P\in\Q[t,x]$ given by
\[P(t,x)=t^4x^3+a(t)x^2+b(t)x+c(t),\]
where
\begin{align*}
a(t) &= 11t^6 + 11t^5 + 11t^4 + 7t^3 + 5t^2 + 2t + 1,\\
b(t) &= 19t^8 + 38t^7 + 57t^6 + 60t^5 + 55t^4 + 38t^3 + 24t^2 + 9t + 3,\\
c(t) &= (3t^3 + 3t^2 + 3t + 1)^2(t^4 + t^3 + t^2 + t + 1).
\end{align*}

The following lemma explains the precise relation between the polynomial $P$ and points of period five for maps of the form $\phi_{k,b}$.

\begin{lem}\label{manes_equation}
Let $F$ be a field of characteristic 0, and let $k\in F\setminus\{0,-1/2\}$ and $b\in F^{\ast}$. Suppose that the map $\phi_{k,b}(z)$ has a point of period five in $F$. Then the polynomial $P(k,x)$ has a root in $F$.
\end{lem}
\begin{proof}
This follows immediately from the proof of \cite[p. 683, Prop. 3]{manes}.
\end{proof}

\begin{lem}\label{manes_poly_exc_set}
The exceptional set $\calE(P)$ is finite.
\end{lem}
\begin{proof}
The polynomial $P$ is irreducible and its Galois group $G$ is isomorphic (as a group of permutations of the roots of $P$) to the symmetric group $S_3$. Up to conjugacy, $G$ has exactly two maximal subgroups, namely the unique subgroup $A$ of order 3 (which is isomorphic to the alternating group $A_3$), and a subgroup $B$ of order 2. The fixed fields of these subgroups, $F_A$ and $F_B$, are therefore extensions of $\Q(t)$ of degrees 2 and 3, respectively. Since $P$ is an irreducible cubic polynomial over $\Q(t)$, it generates a cubic extension of $\Q(t)$; thus, by replacing $B$ with a conjugate subgroup if necessary, we may assume that $F_B$ is generated by $P$. 

Using \textsc{Magma} to compute the group $G$ and fixed field $F_A$, we find that the extension $F_A/\Q(t)$ is generated by a root of the polynomial $x^2-f(t)$, where 
\begin{multline*}
f(t)= 256t^{12} + 256t^{11} + 576t^{10} + 384t^9 + 608t^8 + 336t^7 \\
 +432t^6 + 184t^5 + 193t^4 + 60t^3 + 50t^2 + 8t + 5. 
\end{multline*}
By Theorem \ref{main_hit_thm}, in order to show that the exceptional set of $P$ is finite it suffices to show that the curves defined by $P(t,x)=0$ and $x^2=f(t)$ have only finitely many rational points. The genera of these curves are 4 and 5, respectively, so the result is a consequence of Faltings's theorem.
\end{proof}

We can now prove our main result for this section.

\begin{proof}[Proof of Proposition \ref{manes_specializations_prop}]
Let $S=\calE(P)\cup\{0,-1/2\}$, which is a finite set by Lemma \ref{manes_poly_exc_set}. Suppose that $k\in\Q\setminus S$ and $b\in\Q^{\ast}$. Since $k\notin\calE(P)$, the polynomial $P(k,x)$ is irreducible and its Galois group is isomorphic to $\Gal(P)\cong S_3$. By Proposition \ref{decomposition_group_prop} this implies that the Galois group of $P(k,x)$ acts on the roots of this polynomial in the same way that $S_3$ acts on the set $\{1,2,3\}$.

Let $\calP$ denote the set of primes $p$ such that $P(k,x)$ has a root in $\Q_p$. For $i=1,2,3$, let $U_i$ be the stabilizer of $i$ in $S_3$. The Chebotarev Density Theorem (see \cite[Thm. 2]{berend-bilu} or \cite[Thm. 2.1]{krumm_lgp}) implies that the Dirichlet density of $\calP$ is given by
\[\frac{\vert\bigcup_{i=1}^3U_i\vert}{|S_3|}=\frac{2}{3}.\]
The complement of $\calP$ therefore has density $1/3$. Now, if $p$ is a prime not in $\calP$, then $P(k,x)$ does not have a root in $\Q_p$, and hence, by Lemma \ref{manes_equation}, the map $\phi_{k,b}(z)$ has no point of period five in $\Q_p$. This completes the proof of the proposition.
\end{proof}

\subsection*{Acknowledgements} The first author would like to thank Pierre D\`ebes for several helpful discussions related to the material of \S\ref{hit_section}. The second author would like to thank Claus Fieker and A.-Stephan Elsenhans for 
their help in working with the Galois group implementation of Fieker, which has been extended by Elsenhans.

\end{document}